\documentclass{amsart}
\usepackage{amsmath,amsfonts,mathrsfs,enumerate,graphicx,subfig}
\usepackage{color} 

\newtheorem{theorem}{Theorem}

\newtheorem{proposition}{Proposition}[section]

\newtheorem{lemma}[proposition]{Lemma}

\newtheorem{remark}[proposition]{Remark}

\DeclareMathOperator{\id}{id}
\DeclareMathOperator{\dimaff}{dim_{\mathsf{aff}}}

\DeclareMathOperator{\GL}{GL}
\DeclareMathOperator{\PSL}{PSL}
\DeclareMathOperator{\PGL}{PGL}

\DeclareMathOperator{\SL}{SL}

\newcommand{\triple}[1]{{\left\vert\kern-0.25ex\left\vert\kern-0.25ex\left\vert #1 
    \right\vert\kern-0.25ex\right\vert\kern-0.25ex\right\vert}}
\newcommand{\iii}{\mathtt{i}}
\newcommand{\jjj}{\mathtt{j}}
\newcommand{\kkk}{\mathtt{k}}

\title[An affine IFS with two measures of maximal dimension]{A strongly irreducible affine iterated function system with two invariant measures of maximal dimension}
\author{Ian D. Morris and Cagri Sert}
\address{I. D. Morris: Department of Mathematics, University of Surrey, Guildford, GU2 7XH, UK}
\email{i.morris@surrey.ac.uk }
\address{C. Sert: Department Mathematik, ETH Z\"urich, R\"amistrasse 101, 8092, Z\"urich, Switzerland}
\email{cagri.sert@math.uzh.ch}
\begin{document}

\begin{abstract}
A classical theorem of Hutchinson asserts that if an iterated function system acts on $\mathbb{R}^d$ by similitudes and satisfies the open set condition then it admits a unique self-similar measure with Hausdorff dimension equal to the dimension of the attractor. In the class of measures on the attractor which arise as the projections of shift-invariant measures on the coding space, this self-similar measure is the unique measure of maximal dimension. In the context of affine iterated function systems it is known that there may be multiple shift-invariant measures of maximal dimension if the linear parts of the affinities share a common invariant subspace, or more generally if they preserve a finite union of proper subspaces of $\mathbb{R}^d$. In this note we construct examples where multiple invariant measures of maximal dimension exist even though the linear parts of the affinities do not preserve a finite union of proper subspaces.
\end{abstract}
\maketitle

\section{Introduction}

We recall that an iterated function system is by definition a tuple $(T_1,\ldots,T_N)$ of contracting transformations of some metric space $X$, which in this article will be $\mathbb{R}^d$ equipped with the Euclidean distance. To avoid trivialities it will be assumed throughout this article that $N \geq 2$. If $(T_1,\ldots,T_N)$ is an iterated function system acting on $\mathbb{R}^d$ then it is well-known that there exists a unique nonempty compact set $Z\subset \mathbb{R}^d$ with the property $Z =\bigcup_{i=1}^N T_iZ$, called the \emph{attractor} or \emph{limit set} of the iterated function system. If we define $\Sigma_N:=\{1,\ldots,N\}^{\mathbb{N}}$ with the infinite product topology, there exists moreover a well-defined \emph{coding map} $\Pi \colon \Sigma_N \to \mathbb{R}^d$ characterised by the property
\[\Pi\left[(x_k)_{k=1}^\infty\right] = \lim_{n \to \infty} T_{x_1}\cdots T_{x_n}v\]
for all $v \in \mathbb{R}^d$ and $(x_k)_{k=1}^\infty \in \Sigma_N$, and this coding map is a continuous surjection from $\Sigma_N$ to the attractor.

We recall that $(T_1,\ldots,T_N)$ is said to satisfy the \emph{open set condition} if there exists a nonempty open set $U \subseteq \mathbb{R}^d$ such that the sets $T_1U,\ldots,T_NU$ are pairwise disjoint subsets of $U$; if the same condition holds with a nonempty compact set $X\subseteq \mathbb{R}^d$ instead of an open set $U$ then we say that $(T_1,\ldots,T_N)$ satisfies the 
\emph{strong separation condition}. It is not difficult to show that if the strong separation condition is satisfied then the coding map is a homeomorphism from $\Sigma_N$ to the attractor.

 It was shown recently by D.-J. Feng in \cite{Fe19} that if $\mu$ is an ergodic shift-invariant measure on $\Sigma_N$ and $(T_1,\ldots,T_N)$ is an affine iterated function system then $\Pi_*\mu$ is necessarily exact-dimensional: this means that the limit
\[\lim_{r\to 0} \frac{\log \Pi_*\mu(B_r(v))}{\log r}\]
exists for $\Pi_*\mu$-a.e. $v \in \mathbb{R}^d$ and is $\Pi_*\mu$-almost-everywhere constant, where $B_r(v)$ denotes the open Euclidean ball with centre $v$ and radius $r$. This almost sure value will be called the dimension of the measure $\Pi_*\mu$ and is equal to its upper and lower Hausdorff and packing dimensions, see \cite[\S2]{Fa97}.
If $(T_1,\ldots,T_N)$ satisfies the open set condition and the transformations $T_i$ are similarity transformations, it is a classical result of J.E. Hutchinson \cite{Hu81} that there exists a probability measure on the attractor of $(T_1,\ldots,T_N)$ with dimension equal to that of the attractor; moreover, this measure has the form $\Pi_*\mu$ where $\mu$ is a Bernoulli measure on the coding space $\Sigma_N$. In particular $\mu$ is invariant with respect to the shift transformation $\sigma[(x_k)_{k=1}^\infty]:=(x_{k+1})_{k=1}^\infty$. In the more general context in which the transformations $T_i$ are invertible affine transformations of $\mathbb{R}^d$ it is thus natural to ask when there exists an invariant measure $\mu$ on the coding space which projects to a measure with dimension equal to that of the attractor, and if such a measure exists, how many such measures there might be. This question was posed explicitly, in somewhat differing forms, by D. Gatzouras and Y. Peres in \cite{GaPe97} and by A. K\"aenm\"aki in \cite{Ka04}.

In order to describe progress on the problem of finding measures of maximal dimension for affine iterated function systems it is useful to recall some definitions. We recall that the \emph{singular values} of a real invertible $d\times d$ matrix $A$ are defined to be the positive square roots of the eigenvalues of the positive definite matrix $A^\top A$. We write the singular values of $A$ as $\sigma_1(A),\ldots,\sigma_d(A)$ with the convention $\sigma_1(A)\geq \cdots \geq \sigma_d(A)$. We have $\|A\|=\sigma_1(A)$ and $|\det A|=\sigma_1(A)\cdots \sigma_d(A)$ for all $A \in \GL_d(\mathbb{R})$, where $\|\cdot\|$ denotes the operator norm induced by the Euclidean norm. If $d$ is a positive integer and $s$ a non-negative real number, following \cite{Fa88} we define
\[\varphi^s(A):=\left\{\begin{array}{cl}\sigma_1(A)\cdots \sigma_{\lfloor s\rfloor}(A)\sigma_{\lceil s\rceil}(A)^{s-\lfloor s\rfloor}&\text{if }0 \leq s \leq d,\\
|\det A|^{\frac{s}{d}}&\text{if }s \geq d
\end{array}\right.\]
for all real $d\times d$ matrices $A$. The inequality $\varphi^s(AB) \leq \varphi^s(A)\varphi^s(B)$ is valid for all $s$, $A$ and $B$ and was originally noted in \cite{Fa88}. If $(A_1,\ldots,A_N) \in \GL_d(\mathbb{R})^N$ is given then for each $s\geq 0$ we define the $\varphi^s$-pressure of $(A_1,\ldots,A_N)$ to be the quantity
\[P_{\varphi^s}(A_1,\ldots,A_N):=\lim_{n \to \infty} \frac{1}{n}\log \sum_{i_1,\ldots,i_n=1}^N \varphi^s(A_{i_1}\cdots A_{i_n})\]
which is well-defined by subadditivity. The function $s\mapsto P_{\varphi^s}(A_1,\ldots,A_N)$ is continuous with respect to $s$ for fixed  $(A_1,\ldots,A_N) \in \GL_d(\mathbb{R})^N$. When $(A_1,\ldots,A_N) \in \GL_d(\mathbb{R})^N$ is fixed and has the property that $\max_i\|A_i\|<1$ for some norm on $\mathbb{R}^d$, the function $s\mapsto P_{\varphi^s}(A_1,\ldots,A_N)$  has a unique zero which we call the \emph{affinity dimension} of $(A_1,\ldots,A_N)$. If $(T_1,\ldots,T_N)$ is an iterated function system of the form $T_ix=A_ix+v_i$ where $(A_1,\ldots,A_N) \in \GL_d(\mathbb{R})^N$, we define the affinity dimension of $(T_1,\ldots,T_N)$ to be $\dimaff(A_1,\ldots,A_N)$.

The affinity dimension is always an upper bound for the box dimension of the attractor of $(T_1,\ldots,T_N)$, see \cite{Fa88}. If $\mu$ is an ergodic $\sigma$-invariant measure on $\Sigma_N$ then we define its Lyapunov dimension to be the unique zero of the function $[0,\infty) \to \mathbb{R}$ defined by
\[s \mapsto h(\mu)+\lim_{n \to \infty }\frac{1}{n}\int \log \varphi^s(A_{x_1}\cdots A_{x_n})d\mu\left[(x_k)_{k=1}^\infty\right].\]
The Hausdorff dimension of $\Pi_*\mu$ is always bounded above by the Lyapunov dimension of $\mu$, which is bounded above by the affinity dimension of $(A_1,\ldots,A_N)$, see  \cite{Ka04} and  \cite[\S4]{JoPoSi07}. We say that a shift-invariant measure on $\mu$ is a \emph{$\varphi^s$-equilibrium state} for $(A_1,\ldots,A_N)$ if it satisfies
\[P_{\varphi^s}(A_1,\ldots,A_N)=h(\mu)+\lim_{n \to \infty }\frac{1}{n}\int \log \varphi^s(A_{x_1}\cdots A_{x_n})d\mu\left[(x_k)_{k=1}^\infty\right],\]
and a \emph{K\"aenm\"aki measure} if it is a $\varphi^{s_0}$-equilibrium state with $s_0:=\dimaff(A_1,\ldots,A_N)$. For every $(A_1,\ldots,A_N) \in \GL_d(\mathbb{R})^N$ and $s \geq 0$ there exists at least one $\varphi^s$-equilibrium state for $(A_1,\ldots,A_N)$, a point which we discuss in more detail in \S\ref{se:prelim} below. A shift-invariant measure is a K\"aenm\"aki measure if and only if it has Lyapunov dimension equal to $\dimaff(A_1,\ldots,A_N)$.

In certain highly degenerate cases it is possible for the Hausdorff dimension of the attractor of an iterated function system to exceed the dimension of every invariant measure $\Pi_*\mu$ supported on it, and even to exceed the supremum of the dimensions of such measures: see \cite{DaSi17}. However, in generic cases the attractor of an affine iterated function system has Hausdorff dimension equal to the affinity dimension \cite{BaHoRa19,Fa88,Fe19}, and for generic affine iterated functions it is also the case that every K\"aenm\"aki measure $\mu$ on $\Sigma_N$ projects to a measure $\Pi_*\mu$ on the attractor which has dimension equal to the affinity dimension \cite{Fe19,JoPoSi07} and is fully supported on the attractor \cite{BoMo18}.  
 We refer the reader to the articles cited for the various precise meanings of ``generic'' with respect to which these statements are true. It is therefore of interest to ask how many measures of the form $\Pi_*\mu$  may achieve this maximal dimension value. Since any convex combination of measures with maximal dimension  will also have maximal dimension, we ask specifically how many \emph{pairwise mutually singular} measures of the form  $\Pi_*\mu$  may have dimension equal to that of the attractor, where $\mu$ is shift-invariant. In generic cases this is equivalent to asking how many ergodic K\"aenm\"aki measures a given iterated function system may have. This latter question was first raised by A. K\"aenmaki \cite{Ka04} and is the subject of the present article.
   
Let us say that $(A_1,\ldots,A_N) \in \GL_d(\mathbb{R})^N$ is \emph{reducible} if there exists a nonzero proper subspace $V$ of $\mathbb{R}^d$ such that $A_iV= V$ for every $i=1,\ldots,N$, and otherwise is \emph{irreducible}. We also say that $(A_1,\ldots,A_N)$ is \emph{strongly irreducible} if there does not exist a finite collection $V_1,\ldots,V_m$ of nonzero proper subspaces $V_j$ such that $A_i\left(\cup_{j=1}^m V_j\right) = \cup_{j=1}^mV_j$ for every $i$. We extend the notions of irreducibility and strong irreducibility to subsets of $\GL_d(\mathbb{R})$ in the obvious fashion. It is not difficult to see that a subset of $\GL_d(\mathbb{R})$ is (strongly) irreducible if and only if the subsemigroup of $\GL_d(\mathbb{R})$ which it generates is (strongly) irreducible. We will say that an affine iterated function system $(T_1,\ldots,T_N)$ is (strongly) irreducible if it has the form $T_ix=A_ix+v_i$ where $(A_1,\ldots,A_N)$ is (strongly) irreducible.

It is easy to show that every $(A_1,\ldots,A_N) \in \GL_d(\mathbb{R})^N$ has a unique $\varphi^s$-equilibrium state when $s \geq d$. There exist reducible tuples $(A_1,\ldots,A_N) \in \GL_d(\mathbb{R})^N$ which have as many as $(d-\lfloor s\rfloor) {d\choose \lfloor s\rfloor}=\lceil s\rceil {d \choose \lceil s\rceil}$ mutually singular $\varphi^s$-equilibrium states (see \cite{KaMo18}) and it is believed that this is the maximum possible number of mutually singular $\varphi^s$-equilibrium states for any tuple $(A_1,\ldots,A_N) \in \GL_d(\mathbb{R})^N$. This number is known to be a sharp upper bound for the number of mutually singular $\varphi^s$-equilibrium states  in dimensions up to four \cite{Mo18} and for simultaneously upper triangularisable tuples \cite{KaMo18}, but in the general case the best upper bound which has been obtained so far for the number of mutually singular $\varphi^s$-equilibrium states is ${d \choose \lfloor s\rfloor}{d \choose \lceil s\rceil}$, see \cite{BoMo18}. When $s \in (0,d) \cap \mathbb{Z}$ the maximum possible number of mutually singular $\varphi^s$-equilibrium states can be shown to equal ${d \choose s}$ using the techniques of \cite{FeKa11,KaMo18} although this result does not seem to have been explicitly stated in the literature.

If  $(A_1,\ldots,A_N) \in \GL_d(\mathbb{R})^N$ is irreducible, it was shown by D.-J. Feng and A. K\"aenm\"aki in \cite{FeKa11} that $(A_1,\ldots,A_N)$ has a unique $\varphi^s$-equilibrium state for all $s \in (0,1]$, and their argument easily extends to cover the case $s \in [d-1,d)$. In particular if $(T_1,\ldots,T_N)$ is an irreducible affine iterated function system acting on $\mathbb{R}^2$ then it has a unique K\"aenm\"aki measure. It was shown by the first named author and A. K\"aenm\"aki in \cite{KaMo18} that in three dimensions strong irreducibility is sufficient for the uniqueness of $\varphi^s$-equilibrium states (and hence of K\"aenm\"aki measures) but irreducibility is not. A criterion for uniqueness of $\varphi^s$-equilibrium states in terms of irreducibility and strong irreducibility of successive exterior powers was also given in that article, and is discussed further in \S\ref{se:prelim} below. In dimensions higher than two irreducibility does not suffice for the uniqueness of the K\"aenm\"aki measure: using the arguments of \cite[\S9]{KaMo18} together with the results of \cite{Mo17} one may show that the example
\begin{equation}\label{eq:old-example}A_1:=\begin{pmatrix}0&0&\frac{1}{2}\\ 
\frac{2}{3}&0&0\\
0&\frac{1}{2}&0\end{pmatrix},\qquad
 A_2:=\begin{pmatrix}0&\frac{2}{3}&0\\ 
0&0&\frac{1}{2}\\
\frac{1}{2}&0&0\end{pmatrix}
\end{equation}
is irreducible with $\dimaff(A_1,A_2) \in (1,2)$ and has exactly two ergodic K\"aenm\"aki measures.

These examples leave open the question of whether or not strong irreducibility is sufficient for the uniqueness of $\varphi^s$-equilibrium states and K\"aenm\"aki measures in dimensions higher than three. The purpose of this article is to show that in four dimensions strong irreducibility does \emph{not} suffice for the uniqueness of $\varphi^s$-equilibrium states. We give the following example:
\begin{theorem}\label{th:ex1}
Let $\alpha_1,\alpha_2$ be nonzero real numbers such that $|\alpha_1| \neq |\alpha_2|$ and let $\theta \in \mathbb{R}\setminus \frac{\pi}{2}\mathbb{Z}$. Let $(A_1,A_2) \in \GL_4(\mathbb{R})^2$ where
\[A_1=\begin{pmatrix} \alpha_1 & 0\\ 0&\alpha_2 \end{pmatrix} \otimes \begin{pmatrix} \cos \theta & -\sin \theta\\ \sin \theta &\cos\theta\end{pmatrix}=\begin{pmatrix}\alpha_1 \cos\theta &-\alpha_1\sin\theta & 0 &0\\
\alpha_1\sin\theta & \alpha_1\cos\theta & 0 & 0\\
0&0&\alpha_2\cos\theta & -\alpha_2\sin\theta \\ 
0&0& \alpha_2\sin\theta & \alpha_2\cos\theta
\end{pmatrix}\]
and
\[A_2=\begin{pmatrix} \cos \theta & -\sin \theta\\ \sin \theta &\cos\theta\end{pmatrix}\otimes \begin{pmatrix} \alpha_1 & 0\\ 0&\alpha_2 \end{pmatrix}=\begin{pmatrix}
\alpha_1\cos\theta & 0 & -\alpha_1\sin\theta & 0 \\
0&\alpha_2\cos\theta & 0 &-\alpha_2\sin\theta \\
\alpha_1\sin\theta &0 &\alpha_1\cos\theta & 0 \\
0&\alpha_2\sin\theta & 0&\alpha_2\cos\theta\end{pmatrix}.\]
Then $(A_1,A_2)$ is strongly irreducible and for every $s \in (1,3)$ there exist exactly two distinct ergodic $\varphi^s$-equilibrium states for $(A_1,A_2)$. These equilibrium states are both fully supported on $\Sigma_N$.
\end{theorem}
Here the symbol $A\otimes B$ represents the Kronecker product of the two matrices $A$ and $B$, which is a standard mechanism for representing the tensor product of two linear maps in terms of their matrices; for a more detailed description see \S\ref{se:prelim} below. The fact that the equilibrium states are fully supported will be easily obtained during the proof, but also follows from the far more general results of \cite{BoMo18}.

Theorem \ref{th:ex1} arises as a special case of a substantially more general result whose statement requires some additional notation and definitions; we postpone the statement of this more general theorem until the following section. The precise choice of the two matrices in Theorem \ref{th:ex1} incorporates some arbitrary elements so as to facilitate the proof of Theorem \ref{th:ex2} below. In fact, we will see in the remarks following Theorem \ref{th:main} below that almost any contracting pair of matrices of the form $A_1:=B_1\otimes B_2$, $A_2:=B_2\otimes B_1$ with $B_1,B_2 \in \GL_2(\mathbb{R})$ would suffice just as well.

Theorem \ref{th:ex1} implies the existence of strongly irreducible affine iterated function systems in four dimensions where there exists more than one fully-supported measure on the attractor with maximal dimension:
\begin{theorem}\label{th:ex2}
Let $(B_1,B_2,B_3,B_4)=(A_1,A_1,A_2,A_2) \in \GL_4(\mathbb{R})^4$ where $A_1$ and $A_2$ are as defined in Theorem \ref{th:ex1} with $0<\alpha_2<\alpha_1<\frac{1}{1+\sqrt{\frac{3}{2}}}$, $\alpha_1\alpha_2> \frac{1}{16}$, and arbitrary $\theta \in \mathbb{R} \setminus \frac{\pi}{2}\mathbb{Z}$. Then there exists $(v_1,\ldots,v_4) \in (\mathbb{R}^4)^4$ such that the iterated function system defined by $T_ix:=B_ix+v_i$ satisfies the strong separation condition, has $1<\dimaff(B_1,\ldots,B_4)<2$, and admits two mutually singular invariant measures $m_1:=\Pi_*\mu_1$, $m_2:=\Pi_*\mu_2$ with Hausdorff dimension equal to $\dimaff(B_1,\ldots,B_4)$, each of which is fully supported on the attractor.
\end{theorem}
Examples of affine iterated function systems with two fully-supported measures of maximal dimension were previously constructed in two dimensions by A. K\"aenm\"aki and M. Vilppolainen \cite[Example 6.2]{KaVi10} and by J. Barral and D.-J. Feng \cite{BaFe11}; in these examples the linear parts of the affinities are given by diagonal matrices, and in particular these examples are not irreducible.

Theorem \ref{th:ex1} may be considered to have the following heuristic implication for the investigation of affine iterated function systems. Works which attempt to prove very general statements about the thermodynamic formalism of affine iterated function systems -- that is, assuming only invertibility and perhaps irreducibility of the linear parts of the  system -- can encounter the problem that the number of distinct $\varphi^s$-equilibrium states may in general be very large, forcing any complete mathematical argument to deteriorate into a branching investigation of sub-cases arising from the families of different $\varphi^s$-equilibrium states which may exist for a single iterated function system. (Indeed, in the articles \cite{BoMo18} and \cite{MoSe19} this phenomenon is responsible for most of the length of the proofs of the main results; in those special cases where a unique $\varphi^s$-equilibrium state exists the proofs of the results of both articles can be made an order of magnitude shorter.) It is therefore natural to ask whether a simple, testable general condition can be imposed on an affine IFS which forces the $\varphi^s$-equilibrium state to be unique, allowing simpler and more economical arguments to be applied without any very substantial loss of generality. While it was shown in \cite{KaMo18} that the Zariski density of the semigroup generated by the linear parts is sufficient for the uniqueness of the $\varphi^s$-equilibrium state (see also \cite{JaJaMaLi16} for a closely related result) this condition is arguably the strongest possible condition of an algebraic nature and it is reasonable to ask whether a weaker condition such as strong irreducibility might instead be sufficient. The results of this article demonstrate that this is not the case and strongly suggest that proofs which incorporate the consideration of multiple inequivalent $\varphi^s$-equilibrium states are likely to remain a feature of the literature in situations where statements assuming a weaker condition than Zariski density are proved.

%
%

\section{More general examples}\label{se:two}

Theorem \ref{th:ex1} is obtained as a special case of a more general construction which we now describe. We first establish some necessary notation and definitions. We recall that $\PGL_d(\mathbb{R})$ denotes the quotient of $\GL_d(\mathbb{R})$ by the subgroup consisting of all scalar multiples of the identity matrix. For an element $g \in \GL_d(\mathbb{R})$ we denote by $\bar{g}$ the corresponding equivalence class $\bar{g}\in \PGL_d(\mathbb{R})$. For the purpose of exposition, let $G$ denote either of the groups $\SL_d(\mathbb{R})$, $\GL_d(\mathbb{R})$ and $\PGL_d(\mathbb{R})$. These groups are \emph{linear algebraic groups}: each $G$ can be realised as the set of common zeros of an ideal of polynomials with real coefficients in $k$ variables for some $k \in \mathbb{N}$. 
The Zariski topology on each such group $G$ is defined by declaring the closed subsets of $G$ to be the sets of common zeros of collections of polynomials in $k$ variables. This topology does not depend on the choice of the embedding in the space $\mathbb{R}^k$ and it is coarser than the standard topology on $G$, which we refer to as the  \emph{analytic topology}. For example, $\PGL_d(\mathbb{R})$ is connected with respect to the Zariski topology, whereas for even $d$ it has two connected components with respect to the analytic topology, one corresponding to linear maps with positive determinant (which is equal to $\PSL_d(\mathbb{R})$) and one corresponding to linear maps with negative determinant. A set $Z\subseteq G$ is called \emph{Zariski dense} in $G$ if it is a dense subset of $G$ with respect to the Zariski topology in the usual sense; this is equivalent to the stipulation that every polynomial function $G \to \mathbb{R}$ which is identically zero on $Z$ is also identically zero on $G$. The Zariski closure of any subsemigroup of $G$ is a Lie group with finitely many connected components. In particular a subsemigroup of $G$ fails to be Zariski dense if and only if it is contained in a proper Lie subgroup of $G$ which has finitely many connected components.

If $N \geq 2$ is understood, we will say that a \emph{word} is any finite sequence $\iii=(i_k)_{k=1}^n \in \{1,\ldots,N\}^n$. We define the \emph{length} of the word $\iii=(i_k)_{k=1}^n$ to be $n$ and denote the length of any word $\iii$ by $|\iii|$. We denote the set of all words by $\Sigma_N^*$. If $(A_1,\ldots,A_N) \in \GL_d(\mathbb{R})^N$ is also understood then we define $A_\iii:=A_{i_1}\cdots A_{i_n}$ for every $\iii=(i_k)_{k=1}^n \in \Sigma_N^*$. If $\iota \colon \{1,\ldots,N\} \to \{1,\ldots,N\}$ is a permutation then $\iota$ naturally extends to a map $\iota \colon \Sigma_N^* \to \Sigma_N^*$ defined by $\iota\left[(i_k)_{k=1}^n\right]:=(\iota(i_k))_{k=1}^N$. Clearly $\iota$ thus defined induces a permutation of the set $\{\iii \in \Sigma_N^* \colon |\iii|=n\}$ for each $n \geq1$.

The general result of which Theorem \ref{th:ex1} is a special case is as follows:
\begin{theorem}\label{th:main}
Let $d,N \geq 2$, let $(B_1,\ldots,B_N) \in \GL_d(\mathbb{R})^N$ and suppose that the group generated by the projective linear maps $\bar{B}_i$ is Zariski dense in $\PGL_d(\mathbb{R})$. Let $\iota$ be a permutation of $\{1,\ldots,N\}$ such that:
\begin{enumerate}[(i)]
\item\label{ass1} There does not exist $h \in \PGL_d(\mathbb{R})$  such that for every $\iii \in \Sigma_N^*$, we have $\bar{B}_\iii=h\bar{B}_{\iota(\iii)}h^{-1}$.
\item\label{ass2} There does not exist $h \in \PGL_d(\mathbb{R})$ such that for every $\iii \in \Sigma_N^*$, we have $\bar{B}_\iii=h(\bar{B}_{\iota(\iii)}^{-1})^\top h^{-1}$.
\end{enumerate}
Define an $N$-tuple $(A_1,\ldots,A_N) \in \GL_{d^2}(\mathbb{R})^{N}$ by $A_i:=B_i \otimes B_{\iota(i)}$ for each $i=1,\ldots,N$. Then $(A_1,\ldots,A_N)$ is strongly irreducible and for every $s \in (1,2] \cup [d^2-2,d^2-1)$ there exist exactly two distinct ergodic $\varphi^s$-equilibrium states for $(A_1,\ldots,A_N)$, both of which are both fully supported on $\Sigma_N$. 
\end{theorem}

We make the following remarks:

\begin{remark}
1. Provided that the permutation $\iota$ is non-trivial, the set of tuples $(B_1,\ldots,B_N) \in \GL_d(\mathbb{R})^N$ that do not satisfy the assumptions of Theorem \ref{th:main} is contained in a null set with respect to the Haar measure on $\GL_d(\mathbb{R})^N$. When $\iota$ is non-trivial the assumptions $(i)$ and $(ii)$ hold for an open dense subset of $\GL_d(\mathbb{R})^N$.\\[3pt]
2. In practice, it is easy and usually sufficient to verify the assumptions directly on the generating tuple $(B_1,\ldots,B_N)$, that is, by considering words $\iii$ of length one in the hypotheses (\ref{ass1}) and (\ref{ass2}). \\[3pt]
3. Since the permutation $\iota$ necessarily satisfies $\iota^m=\id$ for some natural number $m \leq N!$ it is not hard to see that if there exists an $h$ satisfying $(\ref{ass1})$ then the corresponding power $h^m$ of $h$ must commute with every element of the semigroup generated by $B_1,\ldots,B_N$. By Zariski density it follows that $h^m$ belongs to the centre of $\PGL_d(\mathbb{R})$, which is the trivial group. We conclude that necessarily $h^m=\id$ for some integer $m$ not greater than $N!$. Similarly if $(\ref{ass2})$ holds then the same observation applies to $(h^{-1})^\top$ in place of $h$. 
\end{remark}


If we define
\[B_1:=\begin{pmatrix} \alpha_1 & 0\\ 0&\alpha_2 \end{pmatrix}, \qquad  B_2:=\begin{pmatrix} \cos \theta & -\sin \theta\\ \sin \theta &\cos\theta\end{pmatrix}\]
where $\alpha_1$, $\alpha_2$ and $\theta$ are as in Theorem \ref{th:ex1} and let $\iota \colon \{1,2\}\to \{1,2\}$ be given by $\iota(1):=2$, $\iota(2):=1$ then clearly the pair $(A_1,A_2)$ defined by Theorem \ref{th:main} corresponds to that considered in Theorem \ref{th:ex1}. Obviously the assumptions $(\ref{ass1})$ and $(\ref{ass2})$ are satisfied. Since $\theta \neq \pi/2$ the group generated by $B_1$ and $B_2$ is a non-elementary subgroup of $\SL_2(\mathbb{R})$ and hence is Zariski dense in $\PGL_2(\mathbb{R})$. In particular Theorem \ref{th:ex1} follows from Theorem \ref{th:main}.

We believe that it should be possible to extend the method of Theorem \ref{th:main} so as to construct examples in dimension $d^2>4$ such that for every $s \in (1,d^2-1)$ there are multiple distinct $\varphi^s$-equilibrium states. In those cases the number of ergodic equilibrium states will in general be much larger than $2$. Such a generalisation of Theorem \ref{th:main} would be likely to need additional hypotheses on the ordering of products of singular values of the matrices $B_i$ in order to make the comparison of different families of equilibrium states practical. Since the purpose of this article is simply to demonstrate that strong irreducibility is compatible with the existence of multiple $\varphi^s$-equilibrium states we do not pursue the problem of optimising Theorem \ref{th:main} in this manner.


The proof of Theorem \ref{th:main} is presented as follows. In the following section we recall some necessary concepts from linear algebra, thermodynamic formalism and the theory of linear algebraic groups. In \S\ref{se:algebraic-bit} we prove that the Zariski density hypothesis of Theorem \ref{th:main} implies the strong irreducibility of the tuple $(A_1,\ldots,A_N)$ and establish some related algebraic facts which will be used later; in \S\ref{se:thm1} we prove the remaining parts of Theorem \ref{th:main} by feeding these algebraic facts into the thermodynamic machinery described in \S\ref{se:prelim}. In \S\ref{se:rmk} we present some further perspectives on the non-uniqueness of equilibrium states in the example presented in Theorem \ref{th:ex1}, and in \S\ref{se:ex2} we give the proof of Theorem \ref{th:ex2}.

%
%

\section{Preliminaries}\label{se:prelim}

\subsection{Linear algebra}\label{subsec.lin.alg}
For the remainder of the article $\|\cdot\|$ will denote either the Euclidean norm defined by the standard inner product or a specified inner product, or the operator norm on matrices defined by such a Euclidean norm. If $A \in \GL_{d_1}(\mathbb{R})$ and $B \in \GL_{d_2}(\mathbb{R})$ are represented by the matrices
\[A=\begin{pmatrix}
a_{11}&\cdots &a_{1{d_1}}\\
\vdots & \ddots & \vdots\\
a_{d_11}& \cdots & a_{d_1d_1}
\end{pmatrix},\qquad 
B=\begin{pmatrix}
b_{11}&\cdots &b_{1d_2}\\
\vdots & \ddots & \vdots\\
b_{d_21}& \cdots & b_{d_2d_2}
\end{pmatrix},
\]
then their Kronecker product may be understood to be the linear map $A\otimes B \in \GL_{d_1d_2}(\mathbb{R})$ with matrix given by
\[A\otimes B = \begin{pmatrix} a_{11}B & \cdots &a_{1d_1} B \\
\vdots & \ddots & \vdots\\
a_{d_11}B& \cdots & a_{d_1d_1}B
\end{pmatrix}.
\]
This construction satisfies the identities $(A_1\otimes B_1)(A_2\otimes B_2)=(A_1A_2)\otimes (B_1B_2)$ and $A^\top \otimes B^\top=(A\otimes B)^\top$  for all $A_1,A_2,A \in \GL_{d_1}(\mathbb{R})$ and $B_1,B_2,B \in \GL_{d_2}(\mathbb{R})$. The identity $(A\otimes B)^{-1}=(A^{-1}\otimes B^{-1})$ follows from the first of these two  identities. If $\alpha_1,\ldots,\alpha_{d_1}$ are the eigenvalues of $A$ and $\alpha_1',\ldots,\alpha_{d_2}'$ the eigenvalues of $B$ then the eigenvalues of $A\otimes B$ are precisely the $d_1d_2$ products $\alpha_i \alpha_j'$ with $1 \leq i \leq d_1$ and $1 \leq j \leq d_2$. Combining these observations  it follows that the singular values of $A\otimes B$ are the products $\sigma_i(A)\sigma_j(B)$ such that $1 \leq i \leq d_1$ and $1 \leq j \leq d_2$ and in particular $\|A\otimes B\|=\sigma_1(A\otimes B)=\sigma_1(A)\sigma_1(B)=\|A\|\cdot\|B\|$ for all $A \in \GL_{d_1}(\mathbb{R})$ and $B \in \GL_{d_2}(\mathbb{R})$. Proofs of these identities may be found in \cite[\S4.2]{HoJo94}. The Kronecker product $A\otimes B$ may be understood algebraically as the matrix representation of the tensor product of the linear maps $A$ and $B$, but this interpretation will not be needed explicitly in the present work.
 
For each $k=1,\ldots,d$ the $k^{\mathrm{th}}$ exterior power of $\mathbb{R}^d$, denoted $\wedge^k\mathbb{R}^d$, is a ${d \choose k}$-dimensional real vector space spanned by the set of all vectors of the form $v_1\wedge v_2\wedge \cdots \wedge v_k$ where $v_1,\ldots,v_k \in \mathbb{R}^d$, where the symbol ``$\wedge$'' is subject to the identities
\[\lambda(v_1 \wedge v_2 \wedge \cdots \wedge v_k) + (v_1' \wedge v_2 \wedge \cdots \wedge v_k)= (\lambda v_1+v_1') \wedge v_2 \wedge \cdots \wedge v_k,\]
\[v_1 \wedge v_2\wedge  \cdots \wedge v_k = (-1)^{i+1}v_i \wedge v_2 \wedge \cdots \wedge v_{i-1}\wedge v_1\wedge v_{i+1}\wedge \cdots \wedge v_k\]
 for all $v_1,\ldots,v_k,v_1' \in \mathbb{R}^d$, $\lambda \in \mathbb{R}$ and $i=1,\ldots,k$.  If $u_1,\ldots,u_d$ is any basis for $\mathbb{R}^d$ then the vectors $u_{i_1}\wedge \cdots \wedge u_{i_k}$ such that $1 \leq i_1<\cdots<i_k\leq d$ form a basis for $\wedge^k\mathbb{R}^d$. The standard inner product $\langle \cdot,\cdot\rangle$ on $\mathbb{R}^d$ induces an inner product on $\wedge^k\mathbb{R}^d$ by
 \[\langle u_1\wedge \cdots \wedge u_k,v_1\wedge \cdots \wedge v_k\rangle:=\det \left[\langle u_i,v_i\rangle\right]_{i,j=1}^k.\]
If $A \in \GL_d(\mathbb{R})$ then $A$ induces a linear map $A^{\wedge k}$ on $\wedge^k\mathbb{R}^d$ by $A^{\wedge k}(u_1\wedge \cdots \wedge u_k)=Au_1 \wedge \cdots \wedge Au_k$. By considering appropriate bases it is easy to see that if $\alpha_1,\ldots,\alpha_d$ are the eigenvalues of $A$ then the eigenvalues of $A^{\wedge k}$ are the numbers $\alpha_{i_1}\cdots \alpha_{i_k}$ such that $1 \leq i_1<\cdots<i_k \leq d$. The identity $(A^\top)^{\wedge k}= (A^{\wedge k})^\top$ follows directly from the definition of the inner product on $\wedge^k\mathbb{R}^d$, and combining these observations we see that the singular values of $A^{\wedge k}$ are precisely the products $\sigma_{i_1}(A)\cdots \sigma_{i_k}(A)$ such that $1 \leq i_1<\cdots<i_k \leq d$. In particular we have $\|\wedge^kA\| \equiv \sigma_1(A^{\wedge k}) \equiv \sigma_1(A)\cdots \sigma_k(A)$. The significance of exterior powers to the present article arises from the following identity: if $A \in \GL_d(\mathbb{R})$ and $0\leq s\leq d$, then
\begin{equation}\label{eq:svd}\varphi^s(A)= \left\|A^{\wedge \lfloor s\rfloor}\right\|^{1+\lfloor s\rfloor -s} \left\|A^{\wedge \lceil s\rceil}\right\|^{s-\lfloor s\rfloor}\end{equation}
by the identity previously remarked.
 
Finally, given $A \in \GL_d(\mathbb{R})$ and $i=1,\ldots,d$ we let $\lambda_i(A)$ denote the modulus of the $i^{th}$ largest of the eigenvalues of $A$. As an easy consequence of the formula $\|A^{\wedge k}\|=\sigma_1(A)\cdots\sigma_k(A)$ together with Gelfand's formula, we have
\begin{equation}\label{eq:yamamoto}\lim_{n\to \infty} \sigma_i(A^n)^{\frac{1}{n}}=\lambda_i(A)\end{equation}
for every $A \in \GL_d(\mathbb{R})$.

\subsection{Linear algebraic groups}\label{subsub.irred}

Given a representation $\rho:\Gamma \to GL_d(\mathbb{R})$ we shall say that $\rho$ is strongly irreducible (resp.\ irreducible) if its image is strongly irreducible (resp.\ irreducible) in the sense previously defined. Given a subgroup $G$ of $\GL_d(\mathbb{R})$ with finitely many connected components which is closed in the analytic topology, it is not difficult to see that $G$ acts strongly irreducibly if and only if the connected component of the identity (in the analytic topology) acts irreducibly. When $G$ is a linear algebraic group, i.e.\ a closed subgroup of $\GL_d(\mathbb{R})$ with respect to the Zariski topology, it automatically has finitely many connected components (for both the analytic and Zariski topologies) and the previous assertion remains true for the connected component of $G$ with respect to the Zariski topology. We also mention that every linear algebraic group is a Lie group, and that the Zariski closure of a semigroup is itself a linear algebraic group. We lastly observe that a subset $S$ of $\GL_d(\mathbb{R})$ remains strongly irreducible/irreducible if each element $g \in S$ is replaced by some nonzero real scalar multiple $cg$, where  $c \in \mathbb{R}^\ast$ is allowed to depend on $g$. Accordingly we may without ambiguity speak of a subset of $\PGL_d(\mathbb{R})$ as being strongly irreducible or irreducible. 
 
In the proof of Theorem \ref{th:main} we will use a special case of a result of Y. Benoist on the properties of the limit cone of a semigroup in a reductive linear algebraic group. For these notions, as well as the following statement, we refer the reader to \cite[Th\'{e}or\`{e}me 1.4]{Be97} (see also \cite[Proposition 1.3]{Qu05} and \cite[Theorem 1.4]{BrSe18}). We observe that given a projective linear transformation $\gamma \in \PGL_d(\mathbb{R})$, the ratios $\frac{\lambda_i}{\lambda_{i+1}}(\gamma)$ of specified pairs of absolute eigenvalues of $\gamma$ are well-defined as the ratio $\frac{\lambda_i(g)}{\lambda_{i+1}(g)}$ for any $g \in \GL_d(\mathbb{R})$ with $\bar{g}=\gamma$. The result which we require is the following:
\begin{proposition}\label{prop.benoist}
Let $\Gamma$ be a Zariski-dense subsemigroup of $\PGL_d(\mathbb{R}) \times \PGL_d(\mathbb{R})$. Then there exists $(\gamma_1,\gamma_2) \in \Gamma$ such that
$$
\frac{\lambda_1}{\lambda_2}(\gamma_1) \neq \frac{\lambda_1}{\lambda_2}(\gamma_2).
$$ 
\end{proposition}

\subsection{Thermodynamic formalism}
If $N \geq 2$ is understood, let $\Sigma_N:=\{1,\ldots,N\}^{\mathbb{N}}$ which we equip with the infinite product topology. This topological space is compact and metrisable. We define the shift transformation $\sigma \colon \Sigma_N \to \Sigma_N$ by $\sigma[(x_k)_{k=1}^\infty]:=(x_{k+1})_{k=1}^\infty$ which is a continuous surjection. We let $\mathcal{M}_\sigma$ denote the set of all $\sigma$-invariant Borel probability measures on $\Sigma_N$ equipped with the weak-* topology, which is the smallest topology such that $\mu \mapsto \int f\,d\mu$ is continuous for every $f \in C(\Sigma_N)$. With respect to this topology $\mathcal{M}_\sigma$ is a nonempty, compact, metrisable topological space.
 
 As was described earlier we will say that a \emph{word} is any finite sequence $\iii=(i_k)_{k=1}^n \in \{1,\ldots,N\}^n$. We define the \emph{length} of the word $\iii=(i_k)_{k=1}^n$ to be $n$ and denote the length of any word $\iii$ by $|\iii|$. When $N$ is understood we denote the set of all words by $\Sigma_N^*$. If $x=(x_k)_{k=1}^\infty \in \Sigma_N$ then we define $x|_n$ to be the word $(x_k)_{k=1}^n \in \Sigma_N^*$. If $\iii \in \Sigma_N^*$ then we define the corresponding \emph{cylinder set} to be the set $[\iii]:=\{x \in \Sigma_N \colon x|_n=\iii\}$. The set of all cylinder sets is a basis for the topology of $\Sigma_N$. If $\iii=(i_k)_{k=1}^n,\jjj=(j_k)_{k=1}^m \in \Sigma_N^*$ are arbitrary words then we define their concatenation $\iii\jjj$ in the obvious fashion: it is the word $(\ell_k)_{k=1}^{n+m}$ such that $\ell_k=i_k$ for $1 \leq k \leq n$ and $\ell_k=j_{k-n}$ for $n+1 \leq k \leq n+m$. If $(A_1,\ldots,A_N) \in \GL_d(\mathbb{R})^N$ is understood then we define $A_\iii:=A_{i_1}\cdots A_{i_n}$ for every $\iii=(i_k)_{k=1}^n \in \Sigma_N^*$. 
  
We will find it convenient in proofs to appeal to more general notions of pressure and equilibrium state than those defined in the introduction. If $N \geq 2$ is understood let us say that a \emph{potential} is any function $\Phi \colon \Sigma_N^* \to (0,+\infty)$. We will say that a potential is \emph{submultiplicative} if it has the property $\Phi(\iii \jjj)\leq \Phi(\iii)\Phi(\jjj)$ for all $\iii,\jjj \in \Sigma_N^*$. All potentials considered in this article will be submultiplicative. If $\Phi$ is a submultiplicative potential then the sequence of functions $\Phi_n \colon \Sigma_N \to (0,+\infty)$ defined by $\Phi_n(x):=\Phi(x|_n)$ satisfies the submultiplicativity relation $\Phi_{n+m}(x) \leq \Phi_n(\sigma^mx)\Phi_m(x)$ for all $n,m \geq 1$ and $x \in \Sigma_N$. Each $\Phi_n$ is continuous since it depends only on finitely many co-ordinates. For every $\mu \in \mathcal{M}_\sigma$ we define
\[\Lambda(\Phi,\mu):=\lim_{n \to \infty} \frac{1}{n}\int \log \Phi(x|_n)d\mu(x) = \inf_{n \geq 1} \frac{1}{n}\int \log \Phi(x|_n)d\mu(x)\]
  which is well-defined by subadditivity. By the subadditive ergodic theorem, if $\mu \in \mathcal{M}_\sigma$ is ergodic then we have $\lim_{n \to \infty}\frac{1}{n}\log\Phi(x|_n)=\Lambda(\Phi,\mu)$ for $\mu$-a.e. $x \in \Sigma_N$.
  
  If $\Phi$ is a submultiplicative potential then we define its \emph{pressure} to be the quantity
 \[P(\Phi):=\lim_{n \to \infty} \frac{1}{n}\log \sum_{|\iii|=n}\Phi(\iii)\] 
which is well-defined by subadditivity.  By the subadditive variational principle of D.-J. Feng, Y.-L. Cao and W. Huang we have 
\[P(\Phi) = \sup_{\mu \in \mathcal{M}_\sigma}\left[ h(\mu)+\Lambda(\Phi,\mu)\right],\]
see \cite[Theorem 1.1]{CaFeHu08}). Since the map $\mu \mapsto \int\log \Phi(x|_n)\,d\mu(x)$ is continuous for each $n \geq 1$ and the map $\mu \mapsto h(\mu)$ is upper semi-continuous, the map $\mu \mapsto h(\mu)+\Lambda(\Phi,\mu)$ is upper semi-continuous. In particular the supremum above is always attained. We call a measure which attains this supremum an equilibrium state for $\Phi$. 
 
 If $(A_1,\ldots,A_N) \in \GL_d(\mathbb{R})$ and $s \geq 0$ then we may define a submultiplicative potential $\Phi^s \colon \Sigma_N \to (0,+\infty)$ by $\Phi^s(\iii):=\varphi^s(A_\iii)$. Clearly in this case $P(\Phi^s)=P_{\varphi^s}(A_1,\ldots,A_N)$ and the notion of equilibrium state for $\Phi^s$ coincides with the notion of $\varphi^s$-equilibrium state for $(A_1,\ldots,A_N)$ introduced in the introduction.  Our mechanism for studying equilibrium states in this article will be the following result which was given as \cite[Corollary 2.2]{BoMo18}:
\begin{theorem}\label{th:x}
Let $\ell \geq 1$ and $N \geq 2$. For each $j=1,\ldots,\ell$ let $d_j \geq 1$ and $\beta_j>0$, and let $(A_1^{(j)},\ldots,A_N^{(j)}) \in \GL_{d_j}(\mathbb{R})^N$ be strongly irreducible. Define a submultiplicative potential $\Phi \colon \Sigma_N^* \to (0,+\infty)$ by
\[\Phi(\iii):=\prod_{j=1}^\ell \left\|A_\iii^{(j)}\right\|^{\beta_j}\]
for all $\iii \in \Sigma_N^*$. Then there exists a unique equilibrium state $\mu$ for $\Phi$. It is ergodic, and there exists a constant $C>0$ such that
\[C^{-1}\Phi(\iii)\leq e^{|\iii|P(\Phi)}\mu([\iii])\leq C\Phi(\iii)\]
for every $\iii \in \Sigma_N^*$.
\end{theorem}

\section{Proof of Theorem \ref{th:main}: the algebraic part}\label{se:algebraic-bit}

In this section we prove that the tuple $(A_1,\ldots,A_N)$ considered in Theorem \ref{th:main} is strongly irreducible and also show that $(B_1^{\wedge 2},\ldots,B_N^{\wedge 2})$ is strongly irreducible, which will be needed later in the proof. Specifically we prove the following statement:
\begin{proposition}\label{pr:key}
Let $(B_1,\ldots,B_N) \in \GL_d(\mathbb{R})^N$ and $\iota$ be as in Theorem \ref{th:main}.  Then $(B_1 \otimes B_{\iota(1)},\ldots,B_N \otimes B_{\iota(N)})$ is strongly irreducible, and for each $k=1,\ldots,d$ the tuple $(B_1^{\wedge k},\ldots,B_N^{\wedge k})$ is strongly irreducible. \end{proposition}

This result will follow by combining the various lemmas given subsequently. We begin with recalling some general results.

\begin{lemma}\label{lemma.irred.zariski}
Let $G$ be a linear algebraic group and suppose that $\pi \colon G \to \GL(V)$ is a strongly irreducible representation. If $\Gamma$ is a Zariski-dense subsemigroup of $G$ then $\pi|_\Gamma$ is a strongly irreducible representation.
\end{lemma}
\begin{proof}
By a previous remark on irreducibility of the (Zariski) connected component $G^o$ (\S \ref{subsub.irred}), it suffices to show that $\pi|_{\Gamma \cap G^o}$ is an irreducible representation. The semigroup $\Gamma \cap G^o$ is clearly Zariski dense in $G^o$ and the latter acts irreducibly on $V$. Since the property of preserving a subspace can be expressed in terms of polynomial equations, the same is true of $\Gamma \cap G^o$ by direct appeal to the definition of Zariski topology.
\end{proof}

We also need the following classical fact (for a proof see e.g.\ \cite{Di51}): 
\begin{lemma}\label{le:isos}
Let $\phi \colon \PGL_d(\mathbb{R}) \to \PGL_d(\mathbb{R})$ be a Lie group automorphism. Then $\phi$ has one of the following two forms: either there exists $x \in \PGL_d(\mathbb{R})$ such that $\phi(g)=xgx^{-1}$ for all $g \in \PGL_d(\mathbb{R})$, or there exists $x \in \PGL_d(\mathbb{R})$ such that $\phi(g)=x(g^\top)^{-1}x^{-1}$ for every $g \in \PGL_d(\mathbb{R})$.
\end{lemma}

In combination with Lemma \ref{lemma.irred.zariski} and Lemma \ref{lemma.tensor.irred} below, the following lemma proves the first statement of Proposition \ref{pr:key}.

\begin{lemma}\label{lemma.zdense}
Let $(B_1,\ldots,B_N) \in \GL_d(\mathbb{R})^N$ and $\iota$ be as in Theorem \ref{th:main}. Then the subsemigroup of $\PGL_d(\mathbb{R}) \times \PGL_d(\mathbb{R})$ generated by $\{(\bar{B}_i, \bar{B}_{\iota(i)}) \colon 1 \leq i \leq N\}$ is Zariski dense in $\PGL_d(\mathbb{R})\times \PGL_d(\mathbb{R})$.
\end{lemma}

\begin{proof}
Let $\Gamma$ denote the subsemigroup of $\PGL_d(\mathbb{R}) \times \PGL_d(\mathbb{R})$ generated by the set $\{(g_i, g_{\iota(i)}) \colon 1 \leq i \leq N\}$. Recall that the Zariski closure of $\Gamma$, which we denote by $G$, is a linear algebraic group.  Let $\pi_1,\pi_2 \colon \PGL_d(\mathbb{R}) \times \PGL_d(\mathbb{R}) \to \PGL_d(\mathbb{R})$ denote the projections onto the first and second co-ordinates respectively. Since the subsemigroup of $\PGL_d(\mathbb{R})$ generated by $g_1,\ldots,g_N$ includes a Zariski dense subsemigroup of $\PGL_d(\mathbb{R})$ it follows that $\pi_1(G)=\pi_2(G)=\PGL_d(\mathbb{R})$. Define
\[H_1:=\ker \pi_1 \subseteq \{\id\} \times \PGL_d(\mathbb{R})\]
and
\[H_2:=\ker \pi_2 \subseteq  \PGL_d(\mathbb{R}) \times \{\id\}\]
and write $H_1=\{\id\}\times N_1$ and $H_2=N_2\times \{\id\}$. Obviously $H_1$ and $H_2$ are normal subgroups of $G$ and are closed in the Zariski topology. Using the surjectivity of the projections $\pi_1$ and $\pi_2$ it is not difficult to deduce that $N_1$ and $N_2$ are also normal subgroups of $\PGL_d(\mathbb{R})$. Consider the map $G \to (\PGL_d(\mathbb{R})/N_1)\times (\PGL_d(\mathbb{R})/N_2)$ defined by $(g,h) \mapsto (gN_1,hN_2)$. By Goursat's lemma the image of this map is the graph of a linear group isomorphism $\phi \colon \PGL_d(\mathbb{R})/N_1 \to \PGL_d(\mathbb{R})/N_2$. 

Our objective is to show that $N_1=N_2=\PGL _d(\mathbb{R})$; once this has been shown it will follow immediately that $G$ contains the groups $\{\id\} \times \PGL_d(\mathbb{R})$ and $\PGL_d(\mathbb{R}) \times \{\id\}$, and hence must equal $\PGL_d(\mathbb{R}) \times \PGL_d(\mathbb{R})$. But the only Zariski-closed normal subgroups of $\PGL_d(\mathbb{R})$ are $\{\id\}$ and $\PGL_d(\mathbb{R})$ itself. Together with the existence of the isomorphism $\phi$ this implies that necessarily $N_1=N_2$ since the two possible quotients of $\PGL_d(\mathbb{R})$ by its normal subgroups are trivially non-isomorphic. We must therefore eliminate the possibility that $N_1=N_2=\{\id\}$.

Suppose for a contradiction that $N_1=N_2=\{\id\}$. In this case we have $G=\{(g,\phi(g)) \colon g \in \PGL_d(\mathbb{R})\}$ where $\phi \colon \PGL_d(\mathbb{R}) \to \PGL_d(\mathbb{R})$ is an isomorphism of linear algebraic groups, and in particular is a Lie group isomorphism. By Lemma \ref{le:isos}, either there exists  $x \in \PGL_d(\mathbb{R})$ such that $\phi(g)= xgx^{-1}$ for all $g \in \PGL_d(\mathbb{R})$, or there exists $x \in \PGL_d(\mathbb{R})$ such that $\phi(g)=x(g^\top)^{-1}x^{-1}$ for all $g \in \PGL_d(\mathbb{R})$; but hypothesis (\ref{ass1}) of Theorem \ref{th:main} excludes the first possibility and hypothesis (\ref{ass2}) excludes the second.  The proof is complete.\end{proof}


In the following two expository results, we note two standard facts and provide brief a proof for the convenience of the readers.

\begin{lemma}\label{lemma.wedge.irred}
Let $1 \leq k \leq d$. Then the representation $\pi \colon \GL_d(\mathbb{R}) \to \GL(\wedge^k \mathbb{R}^d)$ defined by $\pi(g):=g^{\wedge k}$ is strongly irreducible.
\end{lemma}
\begin{proof} It suffices to consider the restriction of $\pi$ to $\SL_d(\mathbb{R})$ and by connectedness of $\SL_d(\mathbb{R})$, it suffices to show that this restriction of $\pi$ is irreducible. To see this, note that any irreducible non-trivial $\SL_d(\mathbb{R})$-submodule (i.e.\ $\SL_d(\mathbb{R})$-invariant non-trivial subspace) $W$ of $\wedge^k\mathbb{R}^d$ is a direct sum of irreducible $A$-submodules, where $A$ is the diagonal subgroup of $\SL_d(\mathbb{R})$. But any irreducible $A$-submodule of $\wedge^d\mathbb{R}^d$ is given by $\mathbb{R} (e_{i_1}\wedge \cdots \wedge  e_{i_k})$ where $e_i$'s is the canonical basis of $\mathbb{R}^d$. Since $\SL_d(\mathbb{R})$ acts transitively on these pure wedge vectors, we have $W=\wedge^k \mathbb{R}^d$ proving the claim.
\end{proof}
One similarly deduces the following  
\begin{lemma}\label{lemma.tensor.irred}
Let $d \geq 1$. Then the representation $\pi \colon \GL_d(\mathbb{R}) \times \GL_d(\mathbb{R}) \to \GL_{d^2}(\mathbb{R})$ defined by $\pi(g,h):=g \otimes h$ is strongly irreducible. \qed
\end{lemma} 

\section{Proof of Theorem \ref{th:main}: the analytic part}\label{se:thm1}

Fix $N$, $d$, $\iota$, $(B_1,\ldots,B_N) \in \GL_d(\mathbb{R})^N$, $(A_1,\ldots,A_N) \in \GL_{d^2}(\mathbb{R})^N$ and $s \in (1,2] \cup [d^2-2,d^2-1)$ as in the statement of Theorem \ref{th:main}. We have $A_{i}:=B_i \otimes B_{\iota(i)}$ for every $i=1,\ldots,N$ and it is clear that $\iota(\iii\jjj)=\iota(\iii)\iota(\jjj)$ and $A_{\iii}=B_\iii \otimes B_{\iota(\iii)}$ for every $\iii,\jjj \in \Sigma_N^*$. We claim that without loss of generality we may make the additional assumption $1 <s \leq 2$. To prove this claim we adapt an argument from \cite[\S2]{Mo18}. Indeed, if $s \in [d^2-2,d^2-1)$ let us define $(B_1',\ldots,B_N') \in GL_d(\mathbb{R})^N$ by 
\[B_i':= |\det B_i|^{\frac{1}{d^2-s}} \left(B_i^{-1}\right)^{\top},\]
define a tuple $(A_1',\ldots,A_N')$ by
\[A_i':=B_i' \otimes B_{\iota(i)}' = |\det A_i|^{\frac{1}{d^2-s}} \left(A_i^{-1}\right)^\top \]
and define $s':=d^2-s \in (1,2]$. The assumptions of Theorem \ref{th:main} are clearly also satisfied by $(B_1',\ldots,B_N')$, $(A_1',\ldots,A_N')$ and $s'$. Furthermore we have
\begin{align*}\varphi^{s'}(A_\iii') &= \sigma_1(A_\iii')\sigma_2(A_\iii')^{s'-1}\\
&= \sigma_1\left(|\det A_\iii |^{\frac{1}{s'}} \left(A_\iii^{-1}\right)^\top \right)\sigma_2\left(|\det A_\iii|^{\frac{1}{s'}} \left(A_\iii^{-1}\right)^\top\right)^{s'-1}\\
 &= |\det A_\iii| \sigma_1\left( A_\iii^{-1} \right)\sigma_2\left( A_\iii^{-1}\right)^{s'-1}\\
  &= |\det A_\iii|  \sigma_{d^2}\left( A_\iii \right)^{-1}\sigma_{d^2-1}\left( A_\iii\right)^{1+s-d^2}\\
  &=\sigma_1(A_\iii)\cdots \sigma_{d^2-1}(A_\iii)^{s-(d^2-2)}=\varphi^s(A_\iii)\end{align*}
  for all $\iii \in\Sigma_N^*$, which implies that the $\varphi^s$-equilibrium states of $(A_1,\ldots,A_N)$ are precisely the $\varphi^{s'}$-equilibrium states of $(A_1',\ldots,A_N')$. Since $1 <s'\leq 2$ we have successfully reduced the general case of Theorem \ref{th:main} to the special case $1<s\leq 2$, proving the claim.

Let us now begin the proof proper, making the additional assumption $1 <s\leq 2$. Define potentials $\Phi$, $\Phi^{(1)}$, $\Phi^{(2)} \colon \Sigma_N^* \to (0,+\infty)$ by
\[\Phi(\iii):=\varphi^s(A_\iii)=\sigma_1(A_\iii)\sigma_2(A_\iii)^{s-1}=\left\|A_\iii\right\|^{2-s} \left\|A_\iii^{\wedge 2}\right\|^{s-1},\]
\[\Phi^{(1)}(\iii):= \sigma_1(B_\iii)^s\sigma_1(B_{\iota(\iii)})\sigma_2(B_{\iota(\iii)})^{s-1}=\left\|B_\iii\right\|^s \left\|B_{\iota(\iii)}\right\|^{2-s} \left\|B_{\iota(\iii)}^{\wedge 2}\right\|^{s-1},\]
\[\Phi^{(2)}(\iii):= \Phi^{(1)}(\iota(\iii))=\sigma_1(B_{\iota(\iii)})^s\sigma_1(B_{\iii})\sigma_2(B_{\iii})^{s-1}=\left\|B_{\iota(\iii)}\right\|^s \left\|B_{\iii}\right\|^{2-s} \left\|B_{\iii}^{\wedge 2}\right\|^{s-1}.\]
It is clear that each is a submultiplicative potential. By Proposition \ref{pr:key}, the tuple $(B_1^{\wedge k},\ldots,B_N^{\wedge k})$ is strongly irreducible for each $k=1,\ldots,d-1$ and obviously so too is $(B_{\iota(1)}^{\wedge k},\ldots,B_{\iota(N)}^{\wedge k})$. Hence the conditions of Theorem \ref{th:x} are met by $\Phi^{(1)}$ and by $\Phi^{(2)}$ and each has a unique equilibrium state. We denote these equilibrium states respectively by $\mu_1$ and $\mu_2$.

We observed in \S\ref{subsec.lin.alg} that for every $A,B \in \GL_d(\mathbb{R})$ the singular values of $A\otimes B$ are precisely the numbers $\sigma_i(A)\sigma_j(B)$ where $1 \leq i,j\leq d$. In particular the largest singular value is $\sigma_1(A)\sigma_1(B)$ and the second-largest is necessarily either $\sigma_1(A)\sigma_2(B)$ or $\sigma_2(A)\sigma_1(B)$. This simple observation implies the identity
\begin{align}\label{eq:very-important}\Phi(\iii) &= \sigma_1(A_\iii)\sigma_2(A_\iii)^{s-1}\\\nonumber
 &= \sigma_1\left(B_\iii \otimes B_{\iota(\iii)}\right)\sigma_2\left(B_\iii \otimes B_{\iota(\iii)}\right)^{s-1}\\\nonumber
&=\sigma_1(B_\iii)\sigma_1(B_{\iota(\iii)})\max\left\{\sigma_1(B_\iii)^{s-1}\sigma_2(B_{\iota(i)})^{s-1},\sigma_1(B_{\iota(\iii)})^{s-1}\sigma_2(B_\iii)^{s-1}\right\}\\\nonumber
&=\max\left\{\Phi^{(1)}(\iii),\Phi^{(2)}(\iii)\right\}\end{align}
which is the fundamental observation around which the whole of Theorem \ref{th:main} is built. 

We claim that $P(\Phi)=P(\Phi^{(1)})=P(\Phi^{(2)})$. Since for every $\iii \in \Sigma_N^*$ we have
\[\sum_{|\iii|=n} \Phi^s(\iii) =\sum_{|\iii|=n} \max\left\{\Phi^{(1)}(\iii),\Phi^{(2)}(\iii)\right\}\leq \left(\sum_{|\iii|=n} \Phi^{(1)}(\iii)+\sum_{|\iii|=n}\Phi^{(2)}(\iii)\right) \]
and
\[\sum_{|\iii|=n} \Phi^s(\iii) =\sum_{|\iii|=n} \max\left\{\Phi^{(1)}(\iii),\Phi^{(2)}(\iii)\right\}\geq \frac{1}{2}\left(\sum_{|\iii|=n} \Phi^{(1)}(\iii)+\sum_{|\iii|=n}\Phi^{(2)}(\iii)\right) \]
it follows by direct consideration of the definition of the pressure that $P(\Phi^s)=\max\{P(\Phi^{(1)}),P(\Phi^{(2)})\}$. To prove the claim it is therefore sufficient to show that $P(\Phi^{(1)})=P(\Phi^{(2)})$. But for every $n \geq 1$ we have
\[\sum_{|\iii|=n}\Phi^{(1)}(\iii)=\sum_{|\iii|=n}\Phi^{(2)}(\iota(\iii))=\sum_{|\iii|=n}\Phi^{(2)}(\iii)\]
where the first equation follows from the definition of $\Phi^{(2)}$ and the second from the fact that $\iota \colon \{1,\ldots,N\}^n \to \{1,\ldots,N\}^n$ is a bijection. The equation $P(\Phi^{(1)})=P(\Phi^{(2)})$ follows directly and the claim is proved.

We now claim that the measures $\mu_1$ and $\mu_2$ are precisely the ergodic equilibrium states of $\Phi$, which is to say the ergodic $\varphi^s$-equilibrium states of $(A_1,\ldots,A_N)$. To see this suppose that $\mu \in \mathcal{M}_\sigma$ is an arbitrary ergodic measure on $\Sigma_N$.  By the subadditive ergodic theorem we have
\begin{align*}\lim_{n \to \infty} \frac{1}{n}\log \Phi^s(x|_n) &= \Lambda\left(\Phi^s,\mu\right),\\
\lim_{n \to \infty} \frac{1}{n}\log \Phi^{(1)}(x|_n) &= \Lambda\left(\Phi^{(1)},\mu\right),\\
\lim_{n \to \infty} \frac{1}{n}\log \Phi^{(2)}(x|_n) &= \Lambda\left(\Phi^{(2)},\mu\right)\end{align*}
for $\mu$-a.e. $x \in \Sigma_N$. In particular for any such $x$ we have
\begin{align*}\Lambda(\Phi^s,\mu)&=\lim_{n \to \infty} \frac{1}{n}\log \Phi^s(x|_n) \\
&=\lim_{n \to \infty} \frac{1}{n}\log \max\left\{\Phi^{(1)}(x|_n),\Phi^{(2)}(x|_n)\right\} \\
&=\max\left\{\lim_{n \to \infty} \frac{1}{n}\log\Phi^{(1)}(x|_n),\lim_{n \to \infty} \frac{1}{n}\log\Phi^{(2)}(x|_n)\right\} \\
&=\max\left\{\Lambda(\Phi^{(1)},\mu),\Lambda(\Phi^{(2)},\mu)\right\}\end{align*}
where we have used \eqref{eq:very-important} in the second equation.
We have shown that  $\Lambda(\Phi^s,\mu)=\max\{\Lambda(\Phi^{(1)},\mu),\Lambda(\Phi^{(2)},\mu)\}$ for every ergodic measure $\mu$. Hence if $\mu$ is an ergodic equilibrium state of $\Phi^{(1)}$ then 
\[P(\Phi^{(1)})=P(\Phi^s)\geq h(\mu)+\Lambda(\Phi^s,\mu) \geq h(\mu)+\Lambda(\Phi^{(1)},\mu)=P(\Phi^{(1)})\]
where the first inequality follows from the subadditive variational principle. It follows that $P(\Phi^s)=h(\mu)+\Lambda(\Phi^s,\mu)$ and therefore $\mu$ is an equilibrium state of $\Phi^s$. Similarly if $\mu$ is an ergodic equilibrium state of $\Phi^{(2)}$ then it is an equilibrium state of $\Phi^s$. On the other hand if $\mu$ is an ergodic equilibrium state of $\Phi^s$ then either $\Lambda(\Phi^s,\mu) = \Lambda(\Phi^{(1)},\mu)$ so that
\[P(\Phi^{(1)})=P(\Phi^s)=h(\mu)+\Lambda(\Phi^s,\mu) = h(\mu)+\Lambda(\Phi^{(1)},\mu)\]
and $\mu$ is an equilibrium state of $\Phi^{(1)}$, or $\Lambda(\Phi^s,\mu) = \Lambda(\Phi^{(2)},\mu)$ so that
\[P(\Phi^{(2)})=P(\Phi^s)=h(\mu)+\Lambda(\Phi^s,\mu) = h(\mu)+\Lambda(\Phi^{(2)},\mu)\]
and $\mu$ is an equilibrium state of $\Phi^{(2)}$. This proves the claim.

We have shown that the ergodic $\varphi^s$-equilibrium states of $(A_1,\ldots,A_N)$ are precisely $\mu_1$ and $\mu_2$, so to complete the proof of the theorem it remains only to show that these two measures are distinct. By Theorem \ref{th:x} there exists $C>0$ such that
\[C^{-1}\Phi^{(1)}(\iii) \leq e^{|\iii| P(\Phi^s)}\mu_1([\iii])=e^{|\iii| P(\Phi^{(1)})}\mu_1([\iii])  \leq C\Phi^{(1)}(\iii) \]
and
\[C^{-1}\Phi^{(2)}(\iii)  \leq e^{|\iii| P(\Phi^s)}\mu_2([\iii])=e^{|\iii| P(\Phi^{(2)})}\mu_2([\iii])  \leq C\Phi^{(2)}(\iii)  \]
for all $\iii \in \Sigma_N^*$. (This shows in particular that both measures are fully supported on $\Sigma_N$, since it implies that every cylinder set has nonzero measure and since cylinder sets form a basis for the topology of $\Sigma_N$.) If it were the case that $\mu_1=\mu_2$ then these inequalities would imply the relation
\begin{equation}\label{eq:potratio}C^{-2}\leq\frac{\Phi^{(1)}(\iii)}{\Phi^{(2)}(\iii)}\leq C^2\end{equation}
for all $\iii \in \Sigma_2^*$. 

By Lemma \ref{lemma.zdense} the subsemigroup of $\PGL_d(\mathbb{R}) \times \PGL_d(\mathbb{R})$ generated by the pairs $(\bar{B}_1,\bar{B}_{\iota(1)}), \ldots, (\bar{B}_N,\bar{B}_{\iota(N)})$ is Zariski dense in $\PGL_d(\mathbb{R}) \times \PGL_d(\mathbb{R})$, so by Proposition \ref{prop.benoist} there exists $\iii_0 \in \Sigma_N^*$ such that $\lambda_1(B_{\iii_0})/\lambda_2(B_{\iii_0}) \neq \lambda_1(B_{\iota(\iii_0)})/\lambda_2(B_{\iota(\iii_0)})$. If we have $\mu_1=\mu_2$ then applying \eqref{eq:potratio} to $\iii:=\iii^n_0$ now yields
\[C^{-2}\leq\left(\frac{\sigma_1\left(B_{\iii_0}^n\right)\sigma_2\left(B_{\iota(\iii_0)}^n\right)}{\sigma_1\left(B_{\iota(\iii_0)}^n\right)\sigma_2\left(B_{\iii_0}^n\right)}\right)^{s-1}\leq C^2\]
for all $n \geq 1$. Since $s - 1\neq 0$ it follows that
\[\lim_{n \to \infty} \left(\frac{\sigma_1\left(B_{\iii_0}^n\right)\sigma_2\left(B_{\iota(\iii_0)}^n\right)}{\sigma_1\left(B_{\iota(\iii_0)}^n\right)\sigma_2\left(B_{\iii_0}^n\right)}\right)^{\frac{1}{n}}=1,\]
but since clearly
\begin{align*}\lim_{n \to \infty}\left(\frac{\sigma_1\left(B_{\iii_0}^n\right)\sigma_2\left(B_{\iota(\iii_0)}^n\right)}{\sigma_1\left(B_{\iota(\iii_0)}^n\right)\sigma_2\left(B_{\iii_0}^n\right)}\right)^{\frac{1}{n}} &=\lim_{n \to \infty} \left(\frac{\sigma_1\left(B_{\iii_0}^n\right)/\sigma_2\left(B_{\iii_0}^n\right)}{\sigma_1\left(B_{\iota(\iii_0)}^n\right)/\sigma_2\left(B_{\iota(\iii_0)}^n\right)}\right)^{\frac{1}{n}} \\
&= \frac{\lambda_1(B_{\iii_0})/\lambda_2(B_{\iii_0})}{\lambda_1(B_{\iota(\iii_0)})/\lambda_2(B_{\iota(\iii_0)})} \neq 1\end{align*}
using \eqref{eq:yamamoto}, this is impossible. We conclude that $\mu_1$ and $\mu_2$ must be distinct, and the theorem is proved.

%
%

\section{Two further perspectives on Theorem \ref{th:main}}\label{se:rmk}
Readers of this article may be aware of the following sufficient condition for the uniqueness of equilibrium states of submultiplicative potentials: a submultiplicative potential $\Phi \colon \Sigma_N^* \to (0,+\infty)$ is called \emph{quasi-multiplicative} if there exist $\delta>0$ and $n_0 \geq 1$ such that
\[\max_{|\kkk| \leq n_0} \Phi(\iii \kkk \jjj )\geq \delta\Phi(\iii)\Phi(\jjj)\]
for all $\iii,\jjj \in \Sigma_N^*$. Every quasi-multiplicative submultiplicative potential has a unique equilibrium state (see for example \cite{Fe11,KaRe14}). Theorem \ref{th:main} therefore implies that the potential $\Phi^s(\iii):=\varphi^s(A_\iii)$ fails to be quasi-multiplicative for $1<s<3$ where $A_1,A_2$ are as defined in the statement of that theorem. 

This failure of quasi-multiplicativity can be demonstrated directly in the following manner. For simplicity of exposition we suppose in this section that $\alpha_1>\alpha_2>0$. Let us consider the second exterior powers of $A_1$ and $A_2$. In the basis $e_1\wedge e_2$, $e_3\wedge e_4$, $e_1 \wedge e_4-e_2\wedge e_3$, $e_1 \wedge e_3$, $e_2\wedge e_4$, $e_1\wedge e_4+e_2\wedge e_3$ for $\wedge^2\mathbb{R}^4$ the matrix of $A_1^{\wedge 2}$ is 
\[\begin{pmatrix}
\alpha_1^2&0&0&0&0&0\\
0&\alpha_2^2&0&0&0&0\\
0&0&\alpha_1\alpha_2&0&0&0\\
0&0&0&\alpha_1\alpha_2\cos^2\theta &\alpha_1\alpha_2\sin^2\theta &-2\alpha_1\alpha_2\cos\theta\sin\theta\\
0&0&0&\alpha_1\alpha_2\sin^2\theta&\alpha_1\alpha_2\cos^2\theta&2\alpha_1\alpha_2\cos\theta\sin\theta\\
0&0&0&\alpha_1\alpha_2\cos\theta\sin\theta&-\alpha_1\alpha_2\cos\theta\sin\theta&\alpha_1\lambda_2(\cos^2\theta-\sin^2\theta)
\end{pmatrix}\]
and that of $A_2^{\wedge 2}$ is
\[\begin{pmatrix}
\alpha_1\alpha_2\cos^2\theta &\alpha_1\alpha_2\sin^2\theta &-2\alpha_1\alpha_2\cos\theta\sin\theta&0&0&0\\
\alpha_1\alpha_2\sin^2\theta&\alpha_1\alpha_2\cos^2\theta&2\alpha_1\alpha_2\cos\theta\sin\theta&0&0&0\\
\alpha_1\alpha_2\cos\theta\sin\theta&-\alpha_1\alpha_2\cos\theta\sin\theta&\alpha_1\lambda_2(\cos^2\theta-\sin^2\theta)&0&0&0\\
0&0&0&\alpha_1^2&0&0\\
0&0&0&0&\alpha_2^2&0\\
0&0&0&0&0&\alpha_1\alpha_2
\end{pmatrix}.\]
In particular if we define
\[B_1:= \begin{pmatrix}
\alpha_1^2&0&0\\
0&\alpha_2^2&0\\
0&0&\alpha_1\alpha_2
\end{pmatrix},\]
\[B_2:= \begin{pmatrix}
\alpha_1\alpha_2\cos^2\theta &\alpha_1\alpha_2\sin^2\theta &-2\alpha_1\alpha_2\cos\theta\sin\theta\\
\alpha_1\alpha_2\sin^2\theta&\alpha_1\alpha_2\cos^2\theta&2\alpha_1\alpha_2\cos\theta\sin\theta\\
\alpha_1\alpha_2\cos\theta\sin\theta&-\alpha_1\alpha_2\cos\theta\sin\theta&\alpha_1\lambda_2(\cos^2\theta-\sin^2\theta)
\end{pmatrix},
 \]
then we have
\[A_1^{\wedge 2}=\begin{pmatrix} B_1&0\\ 0&B_2\end{pmatrix},\qquad A_2^{\wedge 2} =\begin{pmatrix}B_2&0\\0&B_1
\end{pmatrix}\]
in the aforementioned basis. The eigenvalues of $A_1$ are $\alpha_1e^{i\theta}$, $\alpha_1e^{-i\theta}$, $\alpha_2e^{i\theta}$ and $\alpha_2e^{-i\theta}$. The eigenvalues of $A_1^{\wedge 2}$ are the products of pairs of distinct eigenvalues of $A_1$ and hence are $\alpha_1\alpha_2 e^{2i\theta}$, $\alpha_1\alpha_2 e^{-2i\theta}$, $\alpha_1\alpha_2$ (with multiplicity two),  $\alpha_1^2$ and $\alpha_2^2$. Since $B_1$ obviously has eigenvalues $\alpha_1^2$, $\alpha_1\alpha_2$ and $\alpha_2^2$ it follows that the remaining eigenvalues of $A_1^{\wedge 2}$ pertain to $B_2$, and in particular every eigenvalue of $B_2$ has absolute value $\alpha_1\alpha_2$. Thus $\lambda_1(B_1)=\alpha_1^2$ and $\lambda_1(B_2)=\alpha_1\alpha_2$. In particular if $\iii$ is the word consisting of $n$ ones and $\jjj$ the word consisting of $n$ twos, the linear map $A^{\wedge 2}_\iii A^{\wedge2}_\kkk A^{\wedge2}_\jjj$ has the form
\[\begin{pmatrix}
B_1^n & 0\\
0& B_2^n
\end{pmatrix}
\begin{pmatrix}
B_\kkk & 0\\
0& D_\kkk
\end{pmatrix}
\begin{pmatrix}
B_2^n & 0\\
0& B_1^n
\end{pmatrix}.\]
Since $\|B_1^n\|=\alpha_1^{2n}$ and $\|B_2^n\|\simeq \alpha_1^n\alpha_2^n$ the norm of this product is necessarily bounded above by approximately $\alpha_1^{3n}\alpha_2^n \max\{\|B_\kkk\|,\|D_\kkk\|\} \ll \alpha_1^{4n} = \|A^{\wedge 2}_\iii\|\cdot\|A^{\wedge 2}_\jjj\|$. Thus the failure of quasi-multiplicativity of $\varphi^2$, and more broadly of $\varphi^s$ when $1<s<3$, can be seen to arise from the splitting of $\wedge^2\mathbb{R}^4$ into two invariant three-dimensional subspaces on which the actions of $A_1^{\wedge 2}$ and $A_2^{\wedge 2}$ are substantially different.

However, this description of the mechanism of Theorem \ref{th:ex1} in terms of bare-hands algebraic computation is unsatisfying insofar as it lacks any reference to the \emph{a priori} more geometrically relevant action on $\mathbb{R}^4$. We therefore offer the following more geometric explanation of the failure of quasi-multiplicativity for the matrices defined in Theorem \ref{th:main}. Let us define $A_1$-invariant subspaces of $\mathbb{R}^4$ by
\[U_1:=\left\{\begin{pmatrix} a\\ b\\0\\0\end{pmatrix}\colon a,b\in\mathbb{R}\right\}, \qquad U_2:=\left\{\begin{pmatrix} 0\\0\\a\\ b\end{pmatrix}\colon a,b\in\mathbb{R}\right\}\]
and $A_2$-invariant subspaces of $\mathbb{R}^4$ by
\[V_1:=\left\{\begin{pmatrix} a\\ 0\\ b\\0\end{pmatrix}\colon a,b\in\mathbb{R}\right\}, \qquad V_2:=\left\{\begin{pmatrix} 0\\a\\0\\ b\end{pmatrix}\colon a,b\in\mathbb{R}\right\}.\]
We observe that the two larger singular values of $A_1^n$, both being equal to $\alpha_1^n$, arise from its action on  $U_1$, whereas the two smaller singular values are both equal to $\alpha_2^n$ and arise from the invariant subspace $U_2$. Similarly the two largest singular values of $A_2^n$ arise from its action on $V_1$ and the two smaller singular values from its action on $V_2$. In order to find a word $\kkk$ such that the first two singular values of $A_1^nA_\kkk A_2^n$ are both approximately $\alpha_1^{2n}$, therefore, the matrix $A_\kkk$ would have to transpose the subspace $V_1$ into a position where its angle with $U_1$ is bounded away from perpendicularity by an \emph{a priori} amount, and in particular where it does not intersect $U_1^\perp=U_2$. But this is impossible: if $X_1,X_2 \in \GL_2(\mathbb{R})$ are arbitrary matrices then $(X_1 \otimes X_2)V_1$ necessarily intersects $U_2$ nontrivially. To see this let $P$ denote the $2\times 2$ matrix with upper-left entry equal to $1$ and all other entries equal to zero. We have $V_1=(I \otimes P)\mathbb{R}^4$ and $U_2= \ker(P \otimes I)$. If $(X_1 \otimes X_2)V_1$ did not intersect $U_2$ then the image of $(X_1\otimes X_2)(I\otimes P)$ would not intersect the kernel of $P\otimes I$ and the matrices $(P \otimes I)(X_1\otimes X_2)(I\otimes P)=PX_1 \otimes X_2P$ and $(X_1\otimes X_2)(I\otimes P)=X_1 \otimes X_2P$ would have the same rank; but the first matrix has rank one and the second has rank two, since the rank of the Kronecker product of two matrices is equal to the product of their ranks. This argument moreover shows that $U_2 \cap (X_1\otimes X_2)V_1$ has dimension precisely $1$. Thus no element of $\GL_2(\mathbb{R})\otimes \GL_2(\mathbb{R})$, and in particular no element of the group generated by $A_1$ and $A_2$, can move $V_1$ into a position where its intersection with $U_2$ is anything other than one-dimensional.

%
%

\section{Proof of Theorem \ref{th:ex2}}\label{se:ex2}

Suppose that $(A_1,\ldots,A_4) \in \GL_4(\mathbb{R})^4$ satisfies the hypotheses of Theorem \ref{th:main}.  Each $A_i$ is the Kronecker product of a matrix with singular values $\alpha_1$ and $\alpha_2$ and a matrix with singular values $1$ and $1$. Hence each $A_i$ has singular values $\alpha_1$, $\alpha_1$, $\alpha_2$ and $\alpha_2$. Since for every $B \in \GL_4(\mathbb{R})$ we have
\[\sigma_1(B) \geq \left(\sigma_1(B)\sigma_2(B)\sigma_3(B)\sigma_4(B)\right)^{\frac{1}{4}}=|\det B|^{\frac{1}{4}}\]
it follows that for each $n \geq 1$ we have
\[\sum_{|\iii|=n} \varphi^1(A_\iii) \geq \sum_{|\iii|=n} \left|\det A_\iii \right|^{\frac{1}{4}} = \sum_{|\iii|=n} \left(\alpha_1\alpha_2\right)^{\frac{n}{2}}= 4^n\left(\alpha_1\alpha_2\right)^{\frac{n}{2}}\]
which implies that
\[P_{\varphi^1}(A_1,\ldots,A_4)\geq \frac{1}{2}\log \left(16\alpha_1\alpha_2\right)> 0\]
 and therefore $\dimaff(A_1,\ldots,A_4) > 1$. For each $n \geq 1$ we equally have
\[\sum_{|\iii|=n} \varphi^2\left(A_\iii\right) \leq \left(\sum_{i=1}^4 \varphi^2(A_i)\right)^n = \left(4\alpha_1^2\right)^n<\left(\frac{4}{\left(1+\sqrt{\frac{3}{2}}\right)^2}\right)^n\]
where we have made use of the submultiplicativity property $\varphi^2(AB) \leq \varphi^2(A)\varphi^2(B)$ in the first inequality. Hence
\[P_{\varphi^2}(A_1,\ldots,A_4) \leq 2\log(2\alpha_1)<2\log\left(\frac{2}{1+\sqrt{\frac{3}{2}}}\right)<0 \]
so that $\dimaff(A_1,\ldots,A_4) \in (1,2)$ as claimed. For the remainder of the proof define $s:=\dimaff(A_1,\ldots,A_4)$. By Theorem \ref{th:main} there exist precisely two distinct ergodic $\varphi^s$-equilibrium states $\mu_1$, $\mu_2$ for $(A_1,\ldots,A_4)$ and these measures have Lyapunov dimension equal to $\dimaff(A_1,\ldots,A_4)$ and are fully supported on $\Sigma_4$.

Consider now the iterated function system defined by $T_ix:=A_ix+v_i$ for all $x \in \mathbb{R}^4$, where  $(v_1,\ldots,v_4) \in (\mathbb{R}^4)^4$ is to be determined. We claim that the set of all $(v_1,\ldots,v_4) \in (\mathbb{R}^4)^4$ such that $(T_1,\ldots,T_4)$ satisfies the strong separation condition has positive Lebesgue measure. To do this we will show that the set of all such tuples $(v_1,\ldots,v_4)$ contains a nonempty open set. Define
\[v_1:=\begin{pmatrix}1\\0\\\frac{1}{\sqrt{2}}\\0\end{pmatrix},\qquad
v_2:=\begin{pmatrix}-1\\0\\\frac{1}{\sqrt{2}}\\0\end{pmatrix},\qquad
v_3:=\begin{pmatrix}0\\1\\-\frac{1}{\sqrt{2}}\\0\end{pmatrix},\qquad
v_4:=\begin{pmatrix}0\\-1\\-\frac{1}{\sqrt{2}}\\0\end{pmatrix}
 \]
and observe that every two distinct vectors $v_i$, $v_j$ are separated by a Euclidean distance of $2$. Define $X \subset \mathbb{R}^4$ to be the closed origin-centred Euclidean ball of radius $1+\sqrt{\frac{3}{2}}$. For each $i=1,\ldots,4$ the open Euclidean ball of radius $1$ centred on $v_i$ is a subset of $X$, and these subsets do not intersect one another. If we define $T_ix:=A_ix+v_i$ for all $x \in \mathbb{R}^4$ and $i=1,\ldots,4$ then since $\max_{1 \leq i \leq 4}\|A_i\|=\alpha_1<1/(1+\sqrt{\frac{3}{2}})$, each of the sets $T_iX$ is contained in the open Euclidean ball of radius $1$ and centre $v_i$. Since these balls are pairwise disjoint, the sets $T_iX$ are pairwise disjoint subsets of $X$ and therefore $(T_1,\ldots,T_4)$ satisfies the strong separation condition. It is clear that for every $(v_1',\ldots,v_4')$ sufficiently close to $(v_1,\ldots,v_4)$ the four images of $X$ are again contained in the open Euclidean balls of radius $1$ and centre $v_i$, so the strong separation condition remains satisfied for any $(v_1',\ldots,v_4')$ sufficiently close to $(v_1,\ldots,v_4)$. The claim is proved.

We may now prove the theorem. Since
\[\max_{1 \leq i \leq 4} \left\|A_i\right\| =\alpha_1 < \frac{1}{1+\sqrt{\frac{3}{2}}}<\frac{1}{2},\]
by \cite[Theorem 1.9]{JoPoSi07} for Lebesgue a.e. $(v_1,\ldots,v_4) \in (\mathbb{R}^4)^4$ the measures $m_1:=\Pi_*\mu_1$ and $m_2:=\Pi_*\mu_2$ both have dimension equal to their Lyapunov dimension, which is $\dimaff(A_1,\ldots,A_4)$. It follows in particular that there is a positive-measure set of tuples $(v_1,\ldots,v_4)$ such that the strong separation condition is satisfied and additionally $m_1:=\Pi_*\mu_1$ and $m_2:=\Pi_*\mu_2$ both have dimension equal to $\dimaff(A_1,\ldots,A_4)$. When  $(T_1,\ldots,T_4)$ satisfies the strong separation condition we note that $\Pi$ defines a homeomorphism from $\Sigma_4$ to the attractor and therefore $\Pi_*\mu_1$ and $\Pi_*\mu_2$ are mutually singular if and only if $\mu_1$ and $\mu_2$ are; but these two measures are distinct ergodic shift-invariant measures on $\Sigma_4$, and such measures are automatically mutually singular. Since  $\mu_1$ and $\mu_2$ are fully supported on $\Sigma_4$, $\Pi_*\mu_1$ and $\Pi_*\mu_2$ are fully supported on the attractor $\Pi(\Sigma_4)$. The proof is complete.

%
%

\section{Acknowledgements}
The research of I.D. Morris was partially supported by the Leverhulme Trust (Research Project Grant RPG-2016-194). C.S. is supported by SNF grant 178958. The authors thank A. K\"aenm\"aki for several helpful bibliographical suggestions.

\bibliographystyle{acm}
\bibliography{etc}

\end{document}